\documentclass[12pt,letterpaper,reqno]{amsart}
\usepackage{times}
\usepackage[T1]{fontenc}
\usepackage{mathrsfs}
\usepackage{latexsym}
\usepackage[dvips]{graphics}
\usepackage{epsfig}
\usepackage{amsmath,amsfonts,amsthm,amssymb,amscd}
\input amssym.def
\input amssym.tex
\usepackage{color}
\usepackage{hyperref}
\usepackage{url}
\usepackage{breakurl}
\newcommand{\bburl}[1]{\textcolor{blue}{\url{#1}}}

\newcommand{\suppress}[1]{}
\newcommand\be{\begin{equation}}
\newcommand\ee{\end{equation}}
\newcommand\bea{\begin{eqnarray}}
\newcommand\eea{\end{eqnarray}}

\newcommand{\wepsilon}{\widetilde{\epsilon}}

\newtheorem{thm}{Theorem}[section]

\newtheorem{cor}[thm]{Corollary}

\newtheorem{rek}[thm]{Remark}

\begin{document}

\title{Irrationality measure and lower bounds for $\pi(x)$}

\author{David Burt}
\email{\textcolor{blue}{\href{mailto:drb3@williams.edu}{drb3@williams.edu}}}
\address{Department of Mathematics and Statistics, Williams College, Williamstown, MA 01267}

\author[Donow]{Sam Donow}
\email{\textcolor{blue}{\href{mailto:sad3@williams.edu}{sad3@williams.edu}}}
\address{Department of Mathematics and Statistics, Williams College, Williamstown, MA 01267}

\author[Miller]{Steven J. Miller}
\email{\textcolor{blue}{\href{mailto:sjm1@williams.edu}{sjm1@williams.edu}}, \textcolor{blue}{\href{Steven.Miller.MC.96@aya.yale.edu}{Steven.Miller.MC.96@aya.yale.edu}}}
\address{Department of Mathematics and Statistics, Williams College, Williamstown, MA 01267}

\author[Schiffman]{Matthew Schiffman}
\email{\textcolor{blue}{\href{mailto:schiffman.matt@gmail.com}{schiffman.matt@gmail.com}}}
\address{University of Chicago Booth School of Business, Chicago, Illinois 60637}

\author[Wieland]{Ben Wieland}
\email{\textcolor{blue}{\href{mailto:bwieland@gmail.com}{bwieland@gmail.com}}}
\address{Department of Mathematics, Brown University, Providence, RI 02912}

\begin{abstract}
In this note we show how the irrationality measure of $\zeta(s) = \pi^2/6$ can be used to obtain explicit lower bounds for $\pi(x)$. We analyze the key ingredients of the proof of the finiteness of the irrationality measure, and show how to obtain good lower bounds for $\pi(x)$ from these arguments as well. While versions of some of the results here have been done by other authors, our arguments are more elementary and yield a lower bound of order $x/\log x$ as a natural boundary.
\end{abstract}

\subjclass[2010]{11N05 (primary), 11N56, 11J82 (secondary).}

\keywords{Irrationality measure, $\pi(x)$, $\zeta(2)$.}

\date{\today}

\thanks{We thank Emmanuel Kowalski, Tanguy Rivoal and Jonathan Sondow for helpful comments and suggestions on an earlier draft. The third named author was partly supported by NSF grants DMS0600848 and DMS1265673.}

\maketitle

\setcounter{equation}{0}

 %=======================================================
 %  You need the setcounter command above to reset the
 %  equation counting in each section. without it, you'd
 %  have 1.1, 1.2, 1.3, 2.4, 2.5 -- with it (you put it
 %  after each new section) you have 1.1, 1.2, 1.3, 2.1,
 %========================================================

%%%%%%%%%%%%%%%%%%%%%%%%%%%%%%%%%%%%%%%%%%%%%%%%%%%%%%%%%%%%%%%%%%%%%%%%%%%%%%%%%%%%%%%%%%%%%
%%%%%%%%%%%%%%%%%%%%%%%%%%%%%%%%%%%%%%%%%%%%%%%%%%%%%%%%%%%%%%%%%%%%%%%%%%%%%%%%%%%%%%%%%%%%%

\section{Introduction}\setcounter{equation}{0}

One of the most important functions in number theory is $\pi(x)$, the number of primes at most $x$. Many of the proofs of the infinitude of primes fall naturally into one of two categories. First, there are those proofs which provide a lower bound for $\pi(x)$. A classic example of this is Chebyshev's proof (see \cite{Da,MT-B}) that there is a constant $c$ such that $c x / \log x \le \pi(x)$. Another method of proof is to deduce a contradiction from assuming there are only finitely many primes. One of the nicest such arguments is due to Furstenberg (see Chapter 1 of \cite{AZ}), who gives a topological proof of the infinitude of primes. As is often the case with arguments along these lines, we obtain no information about how rapidly $\pi(x)$ grows.

Sometimes proofs which at first appear to belong to one category in fact belong to another. For example, Euclid proved there are infinitely many primes by noting the following: if not, and if $p_1, \dots, p_N$ is a complete enumeration, then either $p_1\cdots p_N+1$ is prime or else it is divisible by a prime not in our list. A little thought shows this proof belongs to the first class, as it yields there are at least $k$ primes at most $2^{2^k}$, thus $\pi(x)\ge \log_2 \log_2 (x)$.

For the other direction, we examine a standard `special value' proof; see \cite{MT-B} for proofs of all the claims below. Consider the Riemann zeta function \be \zeta(s)\ :=\ \sum_{n=1}^{\infty} \frac1{n^s} \ = \ \prod_{p\ {\rm prime}} \left(1 - p^{-s}\right)^{-1}, \ \ee which converges for $\Re s > 1$; the product representation follows from the unique factorization properties of the integers. One can show $\zeta(2) = \pi^2/6$. As $\pi^2$ is irrational, there must be infinitely many primes; if not, the product over primes at $s=2$ would be rational. While at first this argument may appear to belong to the second class (proving $\pi(x)$ tends to infinity without an estimate of its growth), the purpose of this note is to show that it belongs to the first class, and we will obtain an explicit, though \emph{very} weak, lower bound for $\pi(x)$ for all $x$. We deliberately do not attempt to obtain the optimal bounds attainable through this method, but rather concentrate on proving the easiest possible results which best highlight the idea. We conclude by showing how our weak estimates can be fed back into the argument to obtain (infinitely often) massive improvement over the original bounds; our best results here are almost as good as the estimates from Euclid's argument.

%A number $\ga$ is of type $\kappa$ if $\kappa$ is the supremum of
%all $\gamma$ with
%\be\label{eq:irrtypedef}\underline{\lim}_{q\to\infty} q^{\gamma + 1}
%\min_p \left| \ga - \frac{p}{q}\right|\ =\ 0.\ee By Roth's theorem,
%every algebraic irrational is of type $1$. See for example
%\cite{HS,Ro} for more details.

Our lower bounds for $\pi(x)$ use the fact that the irrationality measure of $\pi^2/6$ is bounded. An upper bound on the irrationality measure of an irrational $\alpha$ is a number $ u$ such that there are only finitely many pairs $p$ and $q$ with \be \left| \alpha - \frac{p}{q}\right| \ < \ \frac1{q^{ u}}. \ee The irrationality measure $\mu_{\rm irr}(\alpha)$ is defined to be the infimum of the bounds and need not itself be a bound. Liouville constructed transcendental numbers by studying numbers with infinite irrationality measure, and Roth proved the irrationality measure of an algebraic number is $2$. Currently the best known bound for $\zeta(2)$ is due to Rhin and Viola \cite{RV2}, who give 5.45 as a bound on its irrationality measure. Unfortunately, the published proofs of these bounds use good upper and lower bounds for $d_n := {\rm lcm}(1,\dots,n)$. These upper and lower bounds are obtained by appealing to the Prime Number Theorem (or Chebyshev type bounds); this is a problem for us, as we are trying to prove a weaker version of the Prime Number Theorem (which we are thus subtly assuming in one of our steps!).\footnote{For another example along these lines, see Kowalski \cite{K}. He proves $\pi(x) \gg \log\log x$ by combining the irrationality measure bounds of $\zeta(2)$ with deep results on the distribution of the least prime in arithmetic progressions. Our goal here is to see how far elementary methods can be pushed; in particular, we are trying to see how far one can get without using input about the distribution of primes in progressions. See also \cite{S}, where Sondow proves that $p_{n+1} \le (p_1\cdots p_n)^{2 \mu_{\rm irr}(1/\zeta(2))}$.}

In the arguments below we first examine consequences of the finiteness of the irrationality measure of $\pi^2/6$, deriving lower bounds for $\pi(x)$ in \S\ref{sec:lowerboundspix}. Our best result is Theorem \ref{thm:theorem3}, where we show $\mu_{\rm irr}(\pi^2/6) < \infty$ implies that there is an $M$ such that $\pi(x)\ge \frac{\log\log x}{2\log\log\log x} - M$ infinitely often. We conclude in \S\ref{sec:boundsirrmeasure} by describing how we may modify the standard irrationality measure proofs to yield weaker irrationality bounds which do not require stronger input on $d_n$ then we are assuming. Theorems \ref{thm:theorem2} and \ref{thm:theorem3} are unconditional (explicitly, we may remove the assumption that the irrationality measure of $\pi^2/6$ is finite through a slightly more involved argument). Theorem \ref{thm:theorem4} requires results from Rhin and Viola's \cite{RV2} proof of the irrationality measure, though it only needs weaker results that are independent of the Prime Number Theorem. In Theorem \ref{thm:theorem4} we show that the irrationality measure arguments yield $\pi(x) \ge o(x/\log x)$ for infinitely many $x$ (where $f(x) = o(g(x))$ means that $\lim_{x\to\infty} f(x)/g(x) = 0$), which shows (as one would expect) that $x/\log x$ is a natural boundary for these methods.

%%%%%%%%%%%%%%%%%%%%%%%%%%%%%%%%%%%%%%%%%%%%%%%%%%%%%%%%%%%%%%%%%%%%%%
%%%%%%%%%%%%%%%%%%%%%%%%%%%%%%%%%%%%%%%%%%%%%%%%%%%%%%%%%%%%%%%%%%%%%%

\section{Lower bounds for $\pi(x)$}\label{sec:lowerboundspix} \setcounter{equation}{0}

Define $T(x,k)$ by $T(x,k) = x^{\wedge}(x^{\wedge}(x^{\wedge}(\cdots^{\wedge}x)\cdots))$, with $x$ occurring $k$ times.\\

\begin{thm}\label{thm:theorem1} As $\mu_{\rm irr}(\pi^2/6) < 5.45$, there exists an $N_0$ so that, for all $k$ sufficiently large, \be \pi(T(N_0,2k))\ \ge\ k.\ee
\end{thm}

\begin{proof} For any integer $N$ let $p_N$ and $q_N$ be the relatively prime integers satisfying \be \frac{p_N}{q_N} \ = \ \prod_{p \le N} \left(1 - \frac1{p^2}\right)^{-1} \ = \ \prod_{p \le N} \left(1 + \frac{1}{p^2 - 1}\right). \ \ee \emph{Assume there are no primes $p \in (N,f(N)]$, where $f(x)$ is some rapidly growing function to be determined later.} If $f(N)$ is too large relative to $N$, we will find that $p_N/q_N$ is too good of a rational approximation to $\pi^2/6$, and thus there must be at least one prime between $N$ and $f(N)$. Under our assumption, we find \be \left|\frac{p_N}{q_N} - \frac{\pi^2}6\right| \ = \ \frac{p_N}{q_N} \left|1 - \prod_{p > f(N)} \left(1 + \frac{1}{p^2 - 1}\right)\right|. \ \ee Clearly $p_N/q_N \le \pi^2/6$, and \bea \prod_{p > f(N)} \left(1 + \frac{1}{p^2 - 1}\right) & \ = \ & \exp\left(\log \prod_{p > f(N)} \left(1 + \frac{1}{p^2 - 1}\right) \right) \nonumber\\ &\le & \exp\left( \sum_{n > f(N)} \log\left(1 + \frac1{(n-1)^2}\right) \right) \\ & \le & \exp\left( \sum_{n > f(N)} \frac1{(n-1)^2}\right) \ \le \ \exp\left( \frac{1}{f(N)^2} + \frac{1}{f(N)}\right) \ \eea (the last inequality follows by the replacing the sum over $n \ge f(N)+2$ with an integral). Standard properties of the exponential function yield \be \left|\frac{p_N}{q_N} - \frac{\pi^2}6\right| \ \le \ \frac{\pi^2}{6} \left|1 - \exp\left( \frac{1}{f(N)^2} + \frac{1}{f(N)}\right)\right| \ \le \ \frac{10}{f(N)}. \ \ee The largest $q_N$ can be is $N!^2$, which happens only if all integers at most $N$ are prime. Obviously we can greatly reduce this bound, as the only even prime is $2$; however, our purpose is to highlight the method by using the most elementary arguments possible. If we take $f(x) = (x!)^{14}$, we find (for $N$ sufficiently large) that \be\label{eq:temppisquareddenom} \left| \frac{\pi^2}6 - \frac{p_N}{q_N}\right| \ \le \ \frac{10}{f(N)} \ < \ \frac{1}{q_N^6}; \ee however, this contradicts Rhin and Viola's bound on the irrationality measure of $\pi^2/6$ ($\mu_{\rm irr}(\pi^2/6)$ $<$ $5.45$). Thus there must be a prime between $N$ and $f(N)$. Note $f(N) \le N^{14N} \le (14N)^{14N}$ for large $N$. Letting $f^{(k)}(N)$ denote the result of applying $f$ a total of $k$ times to $N$, for $N_0$ sufficiently large we see for large $k$ that there are at least $k$ primes at most $T(14 N_0,2k)$.
\end{proof}

The inverse of the function $T(N,-)$ is called the $\log^*$ function to base $N$. It is the number of times one can iterate the logarithm without the number becoming non-positive and leaving the domain of the logarithm. It is this extremely slowly growing function that the above theorem yields as a lower bound for $\pi(x)$. The base was determined by the irrationality bound and unspecified (but constructive) bound on the size of the finite number of approximations violating the irrationality bound.

%In the sequel,
%we suppress the base on more ordinary logarithms and so ignore
%multiplicative and additive constants on growth rates.

Of course, this bound arises from assuming that all the numbers at most $x$ are prime (as well as some weak estimation); however, if all the numbers at most $x$ are prime then there are a lot of primes, and we do not need to search for a prime between $N$ and $f(N)$! This interplay suggests that a more careful argument should yield a significantly better estimate on $\pi(x)$, if not for all $x$ then at least infinitely often. We will use an upper bound on $\pi(x)$ with the inequality $q_N\le \prod_{p\le N}(p^2-1)\le N^{2\pi(N)}$. While isolating the true order of magnitude of our bound is difficult, we can easily prove the following.

\begin{thm}\label{thm:theorem2} The finiteness of the irrationality measure of $\pi^2/6$ implies the existence of an $M > 0$ such that for infinitely many $x$ we have $\pi(x) \ge \log\log\log(x) - M$.
\end{thm}

\begin{proof} We choose our constants below to simplify the exposition, and not to obtain the sharpest results. Let $b$ be a bound on the irrationality measure of $\pi^2/6$. The theorem trivially follows if $\pi(x) \ge (\log x)^{e-1} / 4b$ infinitely often, so we may assume that $\pi(x) < (\log x)^{e-1}/4b$ for all $x$ sufficiently large. Thus the denominator $q_N$ in our rational approximation in equation \eqref{eq:temppisquareddenom}, when we consider primes at most $N$ for $N$ sufficiently large, has the bound \be q_N^b\ \le\ N^{2b\pi(N)} \ =\ \exp(2b\pi(N) \log N) \ < \ \exp\left(\frac{(\log N)^e}2\right)\ \le\ \frac{\exp\left(\log (N)^e\right)}{10}.  \ee Thus, if $f(N)=\exp(\log(N)^e)$, we have checked the right-hand inequality of equation \eqref{eq:temppisquareddenom}, which in this case is that $10/f(N) < 1/q_N^b < 10/\exp(\log N)^e$. This cannot hold for $N$ sufficiently large without violating Rhin and Viola's bound on the irrationality measure, unless of course there is a prime between $N$ and $f(N)$. Thus there must be a prime between $N$ and $f(N)$ for all $N$ large.  Define $x_n$ by $x_0=e^e$ and iterating by applying $f$, so that $x_{n+1}=f(x_n)=\exp((\log x_n)^e)$. Then $\log x_{n+1}= (\log x_n)^e$, so $\log x_n = (\log x_0)^{e^n}=\exp e^n$ or $x_n=\exp(\exp e^n)$. Once $x_{M}$ is sufficiently large so that the above argument applies, there is a prime between every pair of $x_i$, so there are at least $n-M$ primes less than $x_n$.
\end{proof}

The simple argument above illustrates how our result can improve itself (at least for an increasing sequence of $x$'s). Namely, the lower bound we obtain is better the fewer primes there are, and if there are many primes we can afford to wait awhile before finding another prime. By more involved arguments, one can show that $\pi(x) \ge h(x)$ infinitely often for many choices of $h(x)$. Sadly, however, none of these arguments allow us to take $h(x) = \log \log x$. Our attempts at obtaining such a weak bound gave us a new appreciation of the estimate from Euclid's argument! Our best result along these lines is the following.

\begin{thm}\label{thm:theorem3} The finiteness of the irrationality measure of $\pi^2/6$ implies the existence of an $M > 0$ such that for infinitely many $x$ we have $\pi(x)\ge \frac{\log\log x}{2\log\log\log x} - M$.
\end{thm}

\begin{proof} The proof is similar to that in Theorem \ref{thm:theorem2}. As before, let $b$ be a bound on the irrationality measure of $\pi^2/6$. We assume that $\pi(x)\le (\log\log x) / 4b$ for all sufficiently large $x$, as otherwise the claim trivially follows. We show that there is a prime between $x_n$ and $x_{n+1}$, where $x_n=\exp(\exp a_n)$ and the sequence $a_n$ is defined by $a_{n+1}=a_n+\log a_n$. It is easy to show that $a_n$ grows like $n\log n$; from there the growth of $x_n$ proves the theorem.  Consider $h(x) = \log\log\log x / \log\log x$. Note $\log^{h(x)} x = \log\log x$, so our assumption can be rewritten as $\pi(x)\le (\log^{h(x)} x)/4b$ for large $x$. Therefore, if $N$ is sufficiently large we have the bound \be q_N^b\le N^{2b\pi(N)}\ =\ \exp(2b\pi(N)\log N)\ \le\ \exp\left(\frac{\log^{h(N)+1} N}2\right) \ \le\ \frac{\exp\left(\log^{h(N)+1} N\right)}{10}. \ee Setting $f(N)=\exp(\log^{h(N)+1}N)$, we see that for large $N$ there must be a prime between $N$ and $f(N)$. We define $x_n$ by iterating $f$ (so $x_{n+1} = f(x_n)$), starting at $x_2=\exp(\exp(e))$.  The recursion can be rewritten as $\log \log x_{n+1}=(h(x_n)+1)\log\log x_n$. In terms of $a_n=\log \log x_n$, this is $a_{n+1}=\left(\frac{\log a_n}{a_n}+1\right)a_n = a_n+\log a_n$. For an upper bound, we have $a_n\le 2n\log n$. We prove this by induction. For the base case, $a_2=e<4\log 2$. If $a_n \le 2n\log n$ with $n \ge 2$, then \be a_{n+1} \ \le \ 2n\log n + \log(2n\log n) \ < \ (2n+1) \log n + \log n \ < \ (2n+2)\log(n+1).  \ \ee For a lower bound, note that $\log a_k\ge1$ so $a_n\ge n$. This improves to $a_{n+1}-a_n=\log a_n \ge \log n$. Therefore $a_{n+1}\ge \sum_{k=1}^n \log k>\int_1^n \log x\,dx =n\log n-n+1$. Thus $n\log n - n < a_n \le 2n\log n$. Therefore $\pi(x_n)\ge n-M$, where $x_M$ is large enough that the assumed bound on $\pi(x_M)$ applies. To derive our asymptotic conclusions, we need to know the inverse of the sequence $x_n$. For $n$ large there are at least $n-M$ primes that are at most $x_n = \exp(\exp a_n) \le \exp(\exp(2n\log n)$. Letting $x = \exp(\exp(2n \log n)$, we find n is at least $\log\log x / 2\log\log\log x$. Therefore, for infinitely many $x$ we have $\pi(x) \ge \log\log x / 2\log\log\log x - M$ (where we subtract $M$ for the same reasons as in Theorem \ref{thm:theorem2}).
\end{proof}

\begin{rek} The lower bound from Theorem \ref{thm:theorem3} is slightly weaker than the one from Euclid's argument, namely that $\pi(x)\ge\log_2\log_2 x$. It is possible to obtain slightly better results by assuming instead that $\pi(x) \le (\log\log x)^{c(x)}$ $/$ $b$; a good choice is to take $c(x) = \log g(x)$ $/$ $\log(g(x) \log g(x))$ with $g(x) = \log \log x / \log \log \log x$. The sequence $a_{n+1} = a_n + \log a_n$ which arises in our proof is interesting, as the Prime Number Theorem states the leading term in the average spacing between primes of size $x$ for large $x$ is $\log x$! Thus $a_n$ is approximately the $n$\textsuperscript{th} prime $p_n$; for example, $a_{1000000} \sim 15479041$ and $p_{1000000} = 15485863$, which differ by about $.044\%$.\end{rek}

%%%%%%%%%%%%%%%%%%%%%%%%%%%%%%%%%%%%%%%%%%%%%%%%%%%%%%%%%%%%%%%%%%%%%%%%%%%%%%%%
%%%%%%%%%%%%%%%%%%%%%%%%%%%%%%%%%%%%%%%%%%%%%%%%%%%%%%%%%%%%%%%%%%%%%%%%%%%%%%%%

\section{Bounds for the irrationality measure of $\pi^2/6$}\label{sec:boundsirrmeasure} \setcounter{equation}{0}

We briefly describe how to modify standard arguments on the irrationality measure of $\zeta(2) = \pi^2/6$ to make Theorems \ref{thm:theorem2} and \ref{thm:theorem3} unconditional. As always, we merely highlight the ideas and do not attempt to prove optimal results. We follow the argument in \cite{RV1}, and by $A(x) = o(B(x))$ we mean $\lim_{x\to \infty} A(x)/B(x) = 0$. With $d_n = {\rm lcm}(1,\dots,n)$, they show the existence of sequences $\{a_n\}$, $\{b_n\}$ such that \be a_n - b_n \zeta(2) \ = \ d_n^2 \int_0^1 \int_0^1 \frac{H_n(x+y,xy)dxdy}{(1-xy)^{n+1}} \ = \ d_n^2 I_n \ \ee for a sequence of polynomials $H_n(u,v)$ with integer coefficients, with $\rho, \sigma > 0$ such that \ \\

\begin{quote}
\begin{quote}\large
\begin{itemize}

\item[(RV1)] $\limsup_{n\to\infty} \frac{\log |b_n|}{n} \le \rho$, and \\ \

\item[(RV2)] $\lim_{n\to\infty} \frac{\log|a_n-b_n\zeta(2)|}{n} = -\sigma$.

\end{itemize}
\end{quote}
\end{quote}\normalsize

\ \\

Then $\mu_{\rm irr}(\zeta(2)) \le 1+\frac{\rho}{\sigma}$ (this is their Lemma 4, and is a special case of Lemma 3.5 in \cite{C}). Unfortunately (for us), they use the Prime Number Theorem to prove that $d_n = \exp(n + o(n))$. From this they deduce that there exist constants $a$ and $b$ such that for any $\epsilon>0$, (i) $\exp((a+2-\epsilon)n)$ $\le d_n^2 I_n$ $\le$ $\exp((a+2+\epsilon)n)$ and (ii) $|b_n|$ $\le$ $\exp((b+2+\epsilon)n)$. Note (i) and (ii) imply (RV1) and (RV2) for our sequences $\{a_n\}$ and $\{b_n\}$ with $\rho = b+2$ and $\sigma = 2-a$, which gives $\mu_{\rm irr}(\zeta(2)) \le (a-b)/(a+2)$. It is very important that the upper and lower bounds of $d_n$ are close, as the limit in (RV2) needs to exist.  We now show how to make Theorems \ref{thm:theorem2} and \ref{thm:theorem3} independent of the Prime Number Theorem (i.e., we do not assume the irrationality measure of $\zeta(2)$ is finite, as the published proofs we know use the Prime Number Theorem). \emph{Assume $\pi(x) \le \log x$ for all $x$ sufficiently large}; if not, then $\pi(x) > \log x$ infinitely often and Theorems \ref{thm:theorem2} and \ref{thm:theorem3} are thus trivial. Under this assumption, we have $1 \le d_n \le \exp(\log^2 n)$. The lower bound is clear. For the upper bound, note the largest power of a prime $p \le n$ that is needed is $\lfloor \log_p n\rfloor \le \log n / \log p$. Thus \be\label{eq:dn} d_n \ \le \ \prod_{p\le n} p^{\log n/\log p} \ = \ \exp\left(\sum_{p \le n} \frac{\log n}{\log p} \cdot \log p\right) \ = \ \exp(\pi(n)\log n); \ee the claimed upper bound follows from our assumption that $\pi(x) \le \log x$. We now find for any $\epsilon > 0$ that (i') $\exp((a-\epsilon)n)$ $\le$ $d_n^2 I_n$ $\le$ $\exp((a+\epsilon)n+2\log^2n)$ and (ii') $|b_n| \le \exp((b+\epsilon)n+2\log^2n)$. We again find that (RV1) and (RV2) hold, and $\mu_{\rm irr}(\zeta(2)) \le (a-b)/a$.

Using the values of $a$ and $b$ from their paper, we obtain (under the assumption that $\pi(x) \le \log x$) that $\mu_{\rm irr}(\zeta(2))$ is finite. Thus Theorems \ref{thm:theorem2} and \ref{thm:theorem3} are independent of the Prime Number Theorem.  Using the values of $a$ and $b$ in \cite{RV1}, we can prove that $\pi(x)$ is quite large infinitely often.

\begin{thm}\label{thm:theorem4} Let $g(x)$ be any function satisfying $g(x) = o(x/\log x)$. Then infinitely often $\pi(x) \ge g(x)$. In particular, for any $\epsilon > 0$ we have $\pi(x) \ge x/\log^{1+\epsilon} x$ infinitely often.
\end{thm}

\begin{proof} We assume $\pi(x) \le g(x)$ for all $x$ sufficiently large, as otherwise the claim is trivial. In \cite{RV1} numerous admissible values of $a$ and $b$ are given (and the determination of these bounds does not use any estimates on the number of primes); we use $a=-2.55306095\ldots$ and $b=1.70036709\ldots$ (page 102). From \eqref{eq:dn} we have $1 \le d_n \le \exp(\pi(n)\log n)$. Using $\pi(x) \le g(x)$ we find (i'') $\exp((a-\epsilon)n)$ $\le$ $d_n^2 I_n$ $\le$ $\exp((a+\epsilon)n+2g(n)\log n)$ and (ii') $|b_n| \le \exp((b+\epsilon)n+2g(n)\log n)$. We again find (RV1) and (RV2) hold, with the same values of $a$ and $b$. For example, to see that (RV2) holds we need to show $\lim_{n\to\infty} (1/n)\log|a_n-b_n\zeta(2)|$ $=$ $-\sigma$. As $a_n - b_n \zeta(2) = d_n^2I_n$, we have for any $\epsilon > 0$ that \be \lim_{n\to\infty} \frac{(a-\epsilon)n}n \ \le \ \lim_{n\to\infty} \frac{\log|a_n-b_n\zeta(2)|}{n} \ \le \ \lim_{n\to\infty} \frac{(a+\epsilon)n+2g(n)\log n}n. \ \ee Our assumption on $g(x)$ implies that $\lim_{n\to\infty} \frac{g(n)\log n}n = 0$, and thus the limit exists as before. We find we may take $\rho = b$ and $\sigma = -a$, which yields $\mu_{\rm irr}(\zeta(2)) \le 1-\frac{b}{a} = 1.666\ldots < 2$. As the irrationality exponent of an irrational number is at least $2$ (see \cite{MT-B} for a proof of this and a proof of the irrationality of $\pi^2$), this is a contradiction. Thus $\pi(x)$ cannot be less than $g(x)$ for all $x$ sufficiently large (and thus infinitely often we beat Euclid by an enormous amount).
\end{proof}

\begin{rek}
We have proved the above in the case of $\pi(x) = o(x/\log x)$. Now, suppose we wanted to get $\pi(x) \sim cx/\log x$ for some $x$. Then, following the calculations above, we would have $b_n \le (b + \epsilon)n + 2g(n\log(n)) = (b + \epsilon)n + 2cn$, so then taking the limit sup as above gives $\rho = b + 2c$. However, if we attempt to take the limit for $\sigma$, we get $(a - \epsilon)n \le d^2 I_n \le \exp((a+\epsilon)n + 2cn)$, and then we can find $a \le \lim_{n\to\infty} (\log \left| a_n - b_n\zeta(2) \right|) / n \le a + 2c$. Notably, the upper and lower bounds are not equal, so we do not know if the limit exists; to show this we would need to have a non-trivial lower bound on $d_n$, which requires the Prime Number Theorem. However, if we had the limit equal to the upper bound, we would have $-\sigma = a + 2c$, and then the irrationality of $\pi^2$ implies $\mu_{\text{irr}}(\zeta(2)) \ge 1 - \frac{b + 2c}{a + 2c}$, which would give us that $c < 0.213$. So, this is possible to show if a lower bound for $d_n$ can be found independent of the Prime Number Theorem. \end{rek}

\begin{rek} It was essential that the limit in (RV2) exist in the above argument. If $\pi(x) \gg x/\log x$ infinitely often and $\pi(x) \ll x/\log^{1+\epsilon} x$ infinitely often then our limit might not exist and we cannot use Lemma 4 of \cite{RV1}. Kowalksi \cite{K} notes\footnote{His note incorrectly mixed up a negation, and the claimed bound of $\pi(x) \gg x^{1-\epsilon}$ is wrong.} that knowledge of $\zeta(s)$ as $s\to 1$ yields $\pi(x) \gg x/\log^{1+\epsilon} x$ infinitely often, which is significantly better than his proof using knowledge of $\zeta(2)$ and Linnik's theorem on the least prime in arithmetic progressions to get $\pi(x) \gg \log\log x$. We may interpret our arguments as correcting this imbalance, as now an analysis of $\zeta(2)$ gives a comparable order of magnitude estimate. It is interesting that the correct growth rate of $\pi(x)$, namely $x/\log x$, surfaces in these arguments as a natural boundary!
\end{rek}

We conclude by improving Theorem \ref{thm:theorem4} to show that not only are we infinitely often close to the true order of growth, but when we are close we are close for large stretches of integers. For notational simplicity we work with logarithms below, but one can easily modify the argument to $o(x/\log x)$.

\begin{cor}\label{thm:theorem5} For any $\wepsilon > 0$ there exists an increasing sequence of numbers $X_{n,\wepsilon}$ tending to infinity such that for each $n$, $\pi(x) \ge x/\log^{1+\wepsilon} x$ for almost all $x \le X_{n,\wepsilon}$.
\end{cor}

\begin{proof} Let $Y_{n,\wepsilon}$ be an increasing sequence tending to infinity so that the result of Theorem \ref{thm:theorem4} holds with exponent $\epsilon = \wepsilon/2$; thus $\pi(Y_{n,\wepsilon}) > Y_{n,\wepsilon}/\log^{1+\wepsilon/2} Y_{n,\wepsilon}$.

Let $X_{n,\wepsilon} = Y_{n,\wepsilon} \log^{\wepsilon/2} Y_{n,\wepsilon}$; we show the claim in the theorem holds for almost all $x \le X_{n,\epsilon}$. We may assume $Y_{n,\wepsilon} \le x \le X_{n,\wepsilon}$, as the fraction of numbers less than $X_{n,\wepsilon}$ which are also less than $Y_{n,\wepsilon}$ tends to zero (the percentage is just $1/\log^{\wepsilon/2} Y_{n,\wepsilon}$).

The claim follows by showing $\pi(x)> x / \log^{1+\wepsilon} x$ for such $x$. We use the fact that $\pi(x)$ is non-decreasing, and by definition of $Y_{n,\wepsilon}$ there are at least $Y_{n,\wepsilon} / \log^{1+\wepsilon/2} Y_{n,\wepsilon}$ primes at most $X_{n,\wepsilon}$. The worst case for us would be that these are all the primes up to $X_{n,\wepsilon}$, but as the number of primes at most $x$ is non-decreasing we have $\pi(x) \ge Y_{n,\wepsilon}/\log^{1+\wepsilon/2} Y_{n,\wepsilon}$ for all $x$ under consideration; we now need to rewrite this in terms of $x$. We have
\begin{eqnarray} \pi(x) \ \ge \ 
\frac{Y_{n,\wepsilon}}{\log^{1+\wepsilon/2} Y_{n,\wepsilon}}  \ \ge \  \frac{x \log^{-\wepsilon/2} Y_{n,\wepsilon}}{\log^{1+\wepsilon/2} Y_{n,\wepsilon}} \ = \ \frac{x}{\log^{1+\wepsilon} Y_{n,\wepsilon}} \ \ge \ \frac{x}{\log^{1+\wepsilon} x};
\end{eqnarray}
thus for $x$ in the desired range we have $\pi(x) \ge x/\log^{1+\wepsilon} x$, completing the proof.
\end{proof}

%%%%%%%%%%%%%%%%%%%%%%%%%%%%%%%%%%%%%%%%%%%%%%%%%%%%%%%%%%%%%%%%%%%%%%%%%%%%%%
%%%%%%%%%%%%%%%%%%%%%%%%%%%%%%%%%%%%%%%%%%%%%%%%%%%%%%%%%%%%%%%%%%%%%%%%%%%%%%

\ \\

\end{document}